\documentclass[a4paper]{amsart}

\usepackage[margin=2cm]{geometry}
\usepackage{amsmath,amsthm,amssymb,amstext}
\usepackage{hyperref}
\usepackage[alphabetic,backrefs]{amsrefs}

\newtheorem{theorem}{Theorem}[section]
\newtheorem{proposition}[theorem]{Proposition}
\newtheorem{lemma}[theorem]{Lemma}
\newtheorem{corollary}[theorem]{Corollary}

\newtheorem{maintheorem}{Theorem}
 
\newtheorem{mainproposition}[maintheorem]{Proposition}

\theoremstyle{remark}

\theoremstyle{definition}
\newtheorem{definition}[theorem]{Definition}

\hypersetup{
 colorlinks=true
}

\begin{document}
\title{
 An asymptotic $\Gamma$-BMT inequality
}
\author{
 Nick Rekuski
}
\address{
 The University of Liverpool, Mathematical Sciences Building, Liverpool L69 7ZL, United Kingdom
}
\email{
 \url{rekuski@liverpool.ac.uk}
}
\date{
 October 22, 2026
}
\begin{abstract}
 We refine a bound of Cheng--Feyzbakhsh to prove an asymptotic $\Gamma$-Bayer--Macr{\`i}--Toda ($\Gamma$-BMT) inequality on every smooth polarized complex threefold.
 The proof relies on a reduction to small tilt stability parameters.
 Using this inequality we verify Feyzbakhsh--Thomas's condition $(\mathbf{wBG})$ whenever $K_{X}$ is numerically proportional to the polarization.
 For polarized Calabi--Yau threefolds, this yields a reduction of positive rank generalized Donaldson--Thomas (DT) invariants to those counting sheaves pure in codimension one.
 We further express these invariants in terms of rank one generalized DT invariants.
 Thus, the rank reductions of Feyzbakhsh--Thomas hold without assuming the BMT conjecture.
\end{abstract}

\maketitle

\section{Introduction}
 Bogomolov--Gieseker (BG) type inequalities for tilt semistable objects are central to understanding Bridgeland stability conditions and curve-counting theories for threefolds.
 The Bayer--Macr{\`i}--Toda (BMT) conjecture~\cite{bayer2014}*{Conjecture 1.3.1} proposes such an inequality involving the third Chern character.
 Although the original inequality fails in general~\cite{schmidt2017}, several modifications retain the main applications.
 For example, the hypothesis $(\mathbf{wBG})$~\cite{feyzbakhsh2024} is sufficient for the reduction of positive rank Donaldson--Thomas invariants to rank zero invariants.
 Likewise, the $\Gamma$-BMT inequality of \cite{bernardara2017} introduces a correction term $\Gamma$ and implies the existence of families of Bridgeland stability conditions on threefolds.
 Earlier results are surveyed in \cite{bayer2023}, see also \cite{feyzbakhsh2025}.

 Our main theorem establishes an asymptotic form of the $\Gamma$-BMT conjecture.
 For tilt stability, we use the parameterization of \cite{feyzbakhsh2023}.

 \begin{maintheorem}[Theorem \ref{theorem:asymptoticGammaBMT}]
  \label{maintheorem}
  Let $(X,H)$ be a smooth polarized complex threefold.
  Set $D=\tfrac{1}{2}K_{X}$.
  There exist $t_{0}>0$ and $\Gamma\in N^{2}(X)_{\mathbb{Q}}$ such that every $\nu_{b,w}^{D}$-semistable $E\in\operatorname{Coh}_{b}^{D}(X)$ with $w-\tfrac{1}{2}b^{2}\geq t_{0}$ and $\nu_{b,w}^{D}(E)=b$ satisfies
  \[
   \operatorname{ch}_{3}^{D+b H}(E)\leq
  \frac{w-\frac{1}{2}b^{2}}{3}H^{2}\cdot\operatorname{ch}_{1}^{D+bH}(E)+\Gamma\cdot\operatorname{ch}_{1}^{D+bH}(E).
   \]
 \end{maintheorem}
 
 Both $t_{0}$ and $\Gamma$ can be chosen independently of $(b,w)$ and $E$.
 Explicit choices of $t_{0}$ and $\Gamma$ are given in Lemma \ref{lemma:mainTheoremText} and Theorem \ref{theorem:asymptoticGammaBMT}.
 
 Our proof of Theorem \ref{maintheorem} adapts the technique of \cite{bayer2016}*{Section 5} to a recent inequality of Cheng--Feyzbakhsh~\cite{cheng2026}.
 It improves the leading coefficient in Cheng--Feyzbakhsh's inequality to the same leading coefficient of the BMT conjecture at the cost of a fixed correction term.
 
 We give two applications of Theorem \ref{maintheorem}.
 First, we prove $(X,H)$ satisfies condition $(\mathbf{wBG})$ of \cite{feyzbakhsh2024}*{Section 3.1} whenever $K_{X}$ is numerically proportional to $H$.
 
 \begin{mainproposition}[Proposition \ref{proposition:threefoldSatisfiesWBG}]
  \label{mainproposition}
  If $(X,H)$ is a smooth polarized complex threefold satisfying $K_{X}\equiv\lambda H$ for some $\lambda\in\mathbb{Q}$ then $(X,H)$ satisfies $(\mathbf{wBG})$.
 \end{mainproposition}
 
 In particular, every polarized Calabi--Yau threefold satisfies $(\mathbf{wBG})$.
 Consequently, \cite{feyzbakhsh2024}*{Theorem 1} holds unconditionally.
 Thus generalized Donaldson--Thomas (DT) invariants counting Gieseker semistable sheaves of positive rank are determined by those counting Gieseker semistable sheaves pure in codimension one.
 
 Our second application further reduces rank zero DT invariants to rank one DT invariants.
 Together with the first application, this establishes \cite{feyzbakhsh2023}*{Theorem 1} without assuming the BMT conjecture.
 In the following statement, $\mathsf{J}(v)\in\mathbb{Q}$ denotes the Joyce--Song generalized Donaldson--Thomas invariant counting $H$-Gieseker semistable sheaves of numerical class $v$~\cite{joyce2012}.
 In addition, we define a polarized Calabi--Yau threefold to be a polarized threefold satisfying $K_{X}\cong\mathcal{O}_{X}$ and $H^{1}(X,\mathcal{O}_{X})=0$.
 
 \begin{maintheorem}[Theorem \ref{theorem:rankOneReduction}]
  \label{mainApplication}
  Suppose $(X,H)$ is a polarized Calabi--Yau threefold.
  For all $v\in K_{\mathrm{num}}(X)$ with $\operatorname{ch}_{0}(v)\geq 0$ there exist numerical classes $\beta_{1},\beta_{2},\ldots,\beta_{m}\in K_{\mathrm{num}}(X)$ with $\operatorname{ch}_{0}(\beta_{i})=1$ and a polynomial $P_{v}\in\mathbb{Q}[t_{1},t_{2},\ldots,t_{m}]$ (depending on $v$, $(X,H)$, and the graded ring $H^{\bullet}(X,\mathbb{Q})$) such that
  \[
   \mathsf{J}(v)=P_{v}\left(\mathsf{J}(\beta_{1}),\mathsf{J}(\beta_{2}),\ldots,\mathsf{J}(\beta_{m})\right).
  \]
 \end{maintheorem}
 
 Our proof of Theorem \ref{mainApplication} follows the strategy of Feyzbakhsh--Thomas~\cite{feyzbakhsh2023}.
 We obtain the necessary wall-crossing estimates from Theorem \ref{maintheorem}, thereby removing the reliance on the BMT conjecture.
 
\subsection*{Relation to other work}
 While this paper was nearing completion, Li--Liu--Perry--Zhao posted a preprint proving the $\Gamma$-BMT conjecture~\cite{li2026b}.
 Their work overlaps with ours.
 In particular, \cite{li2026b}*{Theorem 1.1} implies Theorem \ref{maintheorem}.
 Moreover, \cite{li2026b}*{Theorem 1.3} gives the same application to DT theory.
 The reduction argument in Lemma \ref{lemma:mainTheoremText} also appears in \cite{li2026b}*{Lemma 2.13}.
 Here, we show the asymptotic inequality already suffices for applications to DT theory.
 It would be interesting to determine whether the asymptotic inequality alone suffices for Castelnuovo-type bounds analogous to \cite{li2026b}*{Theorem 1.5}.
 
\subsection*{Notation}
 Throughout, $X$ is a smooth projective complex threefold: a smooth projective integral scheme of dimension three over $\mathbb{C}$.
 We assume $H$ is an integral ample divisor on $X$.
 We refer to the pair $(X,H)$ as a smooth polarized threefold.
 We do not assume that $H$ is primitive.
 
 We write $D$ for an element of $\operatorname{NS}(X)_{\mathbb{Q}}$ and $B$ for an element of $\operatorname{NS}(X)_{\mathbb{R}}$.
 We write $K_{\mathrm{num}}(X)$ for the numerical Grothendieck group and $N^{\bullet}(X)$ for the numerical Chow ring.
 We write $\operatorname{ch}:K_{\mathrm{num}}(X)\to N^{\bullet}(X)_{\mathbb{Q}}$ for the Chern character morphism and use the shorthand $\operatorname{ch}^{B}(E)=\exp(-B)\cdot\operatorname{ch}(E)\in N^{\bullet}(X)_{\mathbb{R}}$.
 
 We write $D^{b}(X)$ for the bounded derived category of coherent sheaves on $X$.
 When discussing tilt stability on $D^{b}(X)$, we work with the parameterization of \cite{feyzbakhsh2023}.
 This differs from the parameterizations of \cites{bayer2014,bayer2016}, \cite{cheng2026}, and \cite{macri2017}; see Table \ref{table:differentParameterizations}.

\subsection*{Acknowledgments}
 The author benefited from conversations with Naoki Koseki, Franco Rota, and Richard Thomas.
 The author also thanks Chunyi Li and Xiaolei Zhao for their encouragement to finish this paper after \cite{li2026b} was posted.
 The author was supported by the Engineering and Physical Sciences Research Council under grant EP/Y009959/1.
 
\subsection*{AI usage}
 The author used generative AI to review this article and to help understand parts of the literature on Donaldson--Thomas theory.

\section{Preliminaries}
 Assume $(X,H)$ is a smooth polarized threefold.
 Fix $D\in\operatorname{NS}(X)_{\mathbb{Q}}$.

\subsection{Slope stability} 
 \begin{definition}
  \label{definition:slopeStable}
  Fix $b\in\mathbb{R}$.
  For $E\in\operatorname{Coh}(X)$ we write
  \[
   \mu_{b}^{D}(E)=
   \begin{cases}
    \frac{H^{2}\cdot\operatorname{ch}_{1}^{D+b H}(E)}{H^{3}\operatorname{ch}_{0}(E)}: & \operatorname{ch}_{0}(E)\neq 0 \\
    +\infty: & \operatorname{ch}_{0}(E)=0
   \end{cases}.
  \]
  We say that a nonzero sheaf $E\in\operatorname{Coh}(X)$ is \emph{$\mu_{b}^{D}$-semistable} if every proper nonzero subobject $F\subseteq E$ satisfies $\mu_{b}^{D}(F)\leq\mu_{b}^{D}(E/F)$.
 \end{definition}
 
 Every nonzero coherent sheaf $E$ admits a \emph{Harder--Narasimhan filtration}, that is, a filtration in $\operatorname{Coh}(X)$:
 \[
  0=E_{0}\subseteq E_{1}\subseteq\cdots\subseteq E_{m-1}\subseteq E_{m}= E
 \]
 whose consecutive factors, called \emph{Harder--Narasimhan factors}, are nonzero $\mu_{b}^{D}$-semistable sheaves of strictly decreasing slope.
 We write $\mu_{b}^{D,+}(E)=\mu_{b}^{D}(E_{1})$ and $\mu_{b}^{D,-}(E)=\mu_{b}^{D}(E_{m}/E_{m-1})$.
 We formally define $\mu_{b}^{D,+}(0)=-\infty$ and $\mu_{b}^{D,-}(0)=+\infty$.
 Since $\mu_{b}^{D}(E)=\mu_{0}^{D}(E)-b$, $\mu_{b}^{D}$-semistability is independent of $b\in\mathbb{R}$.
 
 \begin{definition}
  For $E\in D^{b}(X)$ we define the \emph{discriminant of $E$ (with respect to $D$)} by
  \[
   \Delta^{D}(E)=\left(H^{2}\cdot\operatorname{ch}_{1}^{D}(E)\right)^{2}-2 H^{3}\operatorname{ch}_{0}(E) (H\cdot\operatorname{ch}_{2}^{D}(E)).
  \]
 \end{definition}
 
 The classical Bogomolov--Gieseker inequality and the Hodge index theorem imply $\Delta^{D}(E)\geq 0$ whenever $E$ is $\mu_{b}^{D}$-semistable.

\subsection{Tilt stability} 
 \begin{definition}
  For each $b\in\mathbb{R}$, we define the torsion pair
  \begin{gather*}
   \mathcal{T}_{b}^{D}=\{E\in\operatorname{Coh}(X):\mu_{b}^{D,-}(E)>0\} \\
   \mathcal{F}_{b}^{D}=\{E\in\operatorname{Coh}(X):\mu_{b}^{D,+}(E)\leq 0\}
  \end{gather*}
  with associated tilted heart $\operatorname{Coh}_{b}^{D}(X)=\langle\mathcal{F}_{b}^{D}[1],\mathcal{T}_{b}^{D}\rangle$.
 \end{definition}

 \begin{definition}
  \label{definition:tiltStability}
  Assume $(b,w)\in\mathbb{R}^{2}$ satisfies $w>\tfrac{1}{2}b^{2}$.
  For $E\in D^{b}(X)$ we define
  \[
   \nu_{b,w}^{D}(E)=
   \begin{cases}
    \frac{H\cdot\operatorname{ch}_{2}^{D}(E)-w H^{3}\operatorname{ch}_{0}(E)}{H^{2}\cdot\operatorname{ch}_{1}^{D}(E)-b H^{3}\operatorname{ch}_{0}(E)}: & H^{2}\cdot\operatorname{ch}_{1}^{D}(E)-b H^{3}\operatorname{ch}_{0}(E)\neq 0 \\
    +\infty: & H^{2}\cdot\operatorname{ch}_{1}^{D}(E)-b H^{3}\operatorname{ch}_{0}(E)=0
   \end{cases}.
  \]
  We say a nonzero object $E\in\operatorname{Coh}_{b}^{D}(X)$ is \emph{$\nu_{b,w}^{D}$-semistable} (resp. $\nu_{b,w}^{D}$-stable) if every proper nonzero subobject $F\subseteq E$ satisfies $\nu_{b,w}^{D}(F)\leq\nu_{b,w}^{D}(E/F)$ (resp. $\nu_{b,w}^{D}(F)<\nu_{b,w}^{D}(E/F)$).
 \end{definition}
 
 There are multiple parameterizations of tilt stability in the literature.
 When $D=0$, Definition \ref{definition:tiltStability} agrees with \cites{feyzbakhsh2023,feyzbakhsh2024}.
 Table \ref{table:differentParameterizations} compares our parameterization with those used elsewhere in the literature.

 \begin{table}
  \centering
\setlength{\tabcolsep}{4pt}
 \renewcommand{\arraystretch}{1.35}
 \begin{tabular}{@{}p{.18\linewidth}p{.28\linewidth}p{.3\linewidth}p{.18\linewidth}@{}}
  \textbf{Source}
  & \textbf{Their parameters}
  & \textbf{Our parameters}
  & \textbf{Their setting}\\
  \hline\hline
  \cites{feyzbakhsh2023,feyzbakhsh2024}
  & $\begin{aligned}[t]
      b_{\mathrm{FT}}&=b,\\
      w_{\mathrm{FT}}&=w
     \end{aligned}$
  & $\begin{aligned}[t]
      b&=b_{\mathrm{FT}},\\
      w&=w_{\mathrm{FT}}
     \end{aligned}$
  & $D=0$\\\hline
  \cite{bayer2014},\newline
  \cite{bayer2016}*{Section 2}
  & $\begin{aligned}[t]
      \omega&=\sqrt{6w-3b^{2}}\,H,\\
      B&=D+bH
     \end{aligned}$
  & $\begin{aligned}[t]
      b&=\frac{H^{2}\cdot(B-D)}{H^{3}},\\
      w&=\frac{H\cdot\bigl(\omega^{2}+3(B-D)^{2}\bigr)}{6H^{3}}
     \end{aligned}$
  & $\omega\in\mathbb{R}_{>0}H$,\newline $B-D\in\mathbb{R}H$\\\hline
  \cite{cheng2026}
  & $\begin{aligned}[t]
      & a=\frac{1}{k}\sqrt{6w-3b^{2}+\frac{H\cdot\operatorname{ch}_{2}(\mathcal N)}{2H^{3}}},\\
      & b_{\mathrm{CF}}=\frac{b}{k}
     \end{aligned}$
  & $\begin{aligned}[t]
      b&=kb_{\mathrm{CF}}, \\
      w&=\frac{k^{2}a^{2}}6+\frac{k^{2}b_{\mathrm{CF}}^{2}}2-\frac{H\cdot\operatorname{ch}_{2}(\mathcal N)}{12H^{3}}
     \end{aligned}$
  & $kH$ is very ample,\newline
    $\iota:X\to\mathbb{P}^{N}$ is a closed immersion inducing $kH$ with normal bundle $\mathcal{N}$,\newline
    $D=\tfrac{1}{2}\operatorname{ch}_{1}(\mathcal{N})$\\\hline
  \cite{bayer2016}*{Section 4},\newline
  \cite{macri2017}
  & $\begin{aligned}[t]
      & \alpha=\sqrt{2w-b^{2}}\\
      & \beta=b \\
      & B_{0}=D
     \end{aligned}$
  & $\begin{aligned}[t]
      b&=\beta,\\
      w&=\tfrac{1}{2}(\alpha^{2}+\beta^{2})
     \end{aligned}$
  & \\\hline
 \end{tabular}
  \caption{
   Comparison of parameterizations of tilt stability.
   Subscripts distinguish conflicting notation;
   otherwise, the notation is taken directly from the cited sources.
  }
  \label{table:differentParameterizations}
 \end{table}
 
 The central charge $Z_{b,w}^{D}:K_{\mathrm{num}}(X)\to\mathbb{C}$ associated with Definition \ref{definition:tiltStability} is
 \[
  Z_{b,w}^{D}(E)=-\left(H\cdot\operatorname{ch}_{2}^{D}(E)-w H^{3}\operatorname{ch}_{0}(E)\right)+i\left(H^{2}\cdot\operatorname{ch}_{1}^{D}(E)-b H^{3}\operatorname{ch}_{0}(E)\right).
 \]
 By \cite{bayer2016}*{Proposition 12.2}, $Z_{b,w}^{D}$ is a weak stability function on $\operatorname{Coh}_{b}^{D}(X)$.
 In particular, if $E\in\operatorname{Coh}_{b}^{D}(X)$ then $H^{2}\cdot\operatorname{ch}_{1}^{D+b H}(E)\geq 0$.
 
 We recall some standard facts about $\nu_{b,w}^{D}$-semistability.
 For proofs see \cite{bayer2014}*{Lemma 3.2.1, Lemma 3.2.4, Corollary 7.3.2} and \cite{bayer2016}*{Theorem 3.5, Proposition 12.2, Proposition 12.5}.
 
 \begin{lemma}
  \label{lemma:basicFactsTiltStability}
  Assume $(b,w)\in\mathbb{R}^{2}$ satisfies $w>\tfrac{1}{2}b^{2}$ and fix $E\in\operatorname{Coh}_{b}^{D}(X)$.
  \begin{enumerate}
   \item{
    If $E$ is $\nu_{b,w}^{D}$-semistable then $\Delta^{D}(E)\geq 0$.
   }
   \item{
    $Z_{b,w}^{D}(E)=0$ if and only if $\mathcal{H}^{0}(E)=E$ and $E$ is supported in dimension $\leq 0$.
    Here, $\mathcal{H}^{i}:D^{b}(X)\to\operatorname{Coh}(X)$ is the cohomology functor associated with the standard $t$-structure.
   }
   \item{
    If $E$ is nonzero then there exists a filtration
    \[
     0=E_{0}\subseteq E_{1}\subseteq\cdots\subseteq E_{m}=E
    \]
    whose consecutive factors are nonzero $\nu_{b,w}^{D}$-semistable objects with strictly decreasing $\nu_{b,w}^{D}$.
    We call such a filtration a \emph{$\nu_{b,w}^{D}$-Harder--Narasimhan filtration} and refer to the consecutive quotients as \emph{Harder--Narasimhan factors}.
   }
  \end{enumerate}
 \end{lemma}
 
\subsection{Walls for tilt stability}
 For fixed $D\in\operatorname{NS}(X)_{\mathbb{Q}}$, $\nu_{b,w}^{D}$-stability admits a relatively simple wall-and-chamber decomposition.
 Write
 \[
  U=\{(b,w)\in\mathbb{R}^{2}:w>\tfrac{1}{2}b^{2}\}
 \]
 and for $v\in K_{\mathrm{num}}(X)$ write
 \[
  U_{v}^{D}=\{(b,w)\in U:H^{2}\cdot\operatorname{ch}_{1}^{D}(v)-b H^{3}\operatorname{ch}_{0}(v)>0\}
 \]
 which we equip with the subspace topology.
 For an object $E\in D^{b}(X)$ we also write
 \[
  U_{E}^{D,\mathrm{ss}}=\{(b,w)\in U_{[E]}^{D}:E\in\operatorname{Coh}_{b}^{D}(X),E \ \text{is} \ \nu_{b,w}^{D}-\text{semistable}\}.
 \]
 where $[E]\in K_{\mathrm{num}}(X)$ is the numerical class of $E$.
 We equip $U_{E}^{D,\mathrm{ss}}$ with the subspace topology from $U_{[E]}^{D}$.
 
 We recall some standard facts about the wall-and-chamber decomposition of $U_{v}^{D}$;
 see \cite{bayer2016}*{Corollary 3.10, Proposition 12.2, and Proposition 12.5}, \cite{feyzbakhsh2023}*{Proposition 1.4}, and \cite{feyzbakhsh2024}*{Lemma 3.3}.
 
 \begin{lemma}
  \label{lemma:wallProperties}
  Fix a nonzero class $v\in K_{\mathrm{num}}(X)$ with $\Delta^{D}(v)\geq 0$ and $U_{v}^{D}\neq\emptyset$.
  There exists a locally finite collection $\{\ell_{i}\subseteq U_{v}^{D}\}_{i\in I}$ of connected one-dimensional real submanifolds such that the following statements hold for all $E\in D^{b}(X)$ with $[E]=v$:
  \begin{enumerate}
   \item{
    If $V\subseteq U_{v}^{D}\setminus\bigcup_{i\in I}\ell_{i}$ is a connected component then either $V\subseteq U_{E}^{D,\mathrm{ss}}$ or $V\cap U_{E}^{D,\mathrm{ss}}=\emptyset$.
   }
   \item{
    $U_{E}^{D,\mathrm{ss}}\subseteq U_{[E]}^{D}$ is closed.
   }
   \item{
    If $(b,w)\in U_{E}^{D,\mathrm{ss}}\setminus\operatorname{int}(U_{E}^{D,\mathrm{ss}})$ then there exists a short exact sequence in $\operatorname{Coh}_{b}^{D}(X)$:
    \[
     0\to F\to E\to G\to 0
    \]
    satisfying:
    \begin{itemize}
     \item{
      $F,G$ are nonzero $\nu_{b,w}^{D}$-semistable objects,
     }
     \item{
      $\nu_{b,w}^{D}(F)=\nu_{b,w}^{D}(E)=\nu_{b,w}^{D}(G)$, and
     }
     \item{
      $0\leq\Delta^{D}(F),\Delta^{D}(G)<\Delta^{D}(E)$.
     }
    \end{itemize}
   }
   \item{
    Fix $(b_{*},w_{*})\in U$.
    The set
    \[
     \{i\in I:\text{there exists} \ (b_{*},w)\in U \  \text{with} \ w\geq w_{*} \ \text{and} \ (b_{*},w)\in\ell_{i}\}
    \]
    is finite.
   }
  \end{enumerate}
 \end{lemma}
 
 In the setting of Lemma \ref{lemma:wallProperties}, we call each $\ell_{i}$ a \emph{wall} and each connected component $V\subseteq U_{v}^{D}\setminus\cup_{i\in I}\ell_{i}$ a \emph{chamber}.
 
 For calculations, it is useful to introduce a weaker notion, which we call a numerical line.
 It is even weaker than the standard notion of a numerical wall, since it need not be cut out by an equation $\nu_{b,w}^{D}(v)=\nu_{b,w}^{D}(v^{\prime})$ for numerical classes $v,v^{\prime}$.
 
 \begin{definition}
  \label{definition:numericalLine}
  Fix $v\in K_{\mathrm{num}}(X)$ with $\Delta^{D}(v)\geq 0$ and $U_{v}^{D}\neq\emptyset$.
  A \emph{numerical line for $v$} is a line $\ell\subseteq\mathbb{R}^{2}$ satisfying $\ell\cap U_{v}^{D}\neq\emptyset$ and the following conditions:
  \begin{enumerate}
   \item{
    If $\operatorname{ch}_{0}(v)\neq 0$ then $\ell$ passes through
    \[
     \Pi(v)=\left(\frac{H^{2}\cdot\operatorname{ch}_{1}^{D}(v)}{H^{3}\operatorname{ch}_{0}(v)},\frac{H\cdot\operatorname{ch}_{2}^{D}(v)}{H^{3}\operatorname{ch}_{0}(v)}\right).
    \]
   }
   \item{
    If $\operatorname{ch}_{0}(v)=0$ then $\ell$ has slope
    \[
     \frac{H\cdot\operatorname{ch}_{2}^{D}(v)}{H^{2}\cdot\operatorname{ch}_{1}^{D}(v)}.
    \]
   }
  \end{enumerate}
  
  For a numerical line $\ell$, we also write
  \begin{gather*}
   b_{1}(\ell)=\min\{b\in\mathbb{R}:\text{there exists} \ w\geq\tfrac{1}{2}b^{2} \ \text{with} \ (b,w)\in\ell\},\\
   b_{2}(\ell)=\max\{b\in\mathbb{R}:\text{there exists} \ w\geq\tfrac{1}{2}b^{2} \ \text{with} \ (b,w)\in\ell\}.
  \end{gather*}
 \end{definition}
 
 The following lemma collects some standard facts and calculations regarding numerical lines.
 
 \begin{lemma}
  \label{lemma:factsAboutNumericalLines}
  Fix $v\in K_{\mathrm{num}}(X)$ with $\Delta^{D}(v)\geq 0$ and $U_{v}^{D}\neq\emptyset$.
  \begin{enumerate}
   \item{
    If $\ell$ is a numerical line for $v$ then
    \[
     2 H\cdot\operatorname{ch}_{2}^{D}(v)=(b_{1}(\ell)+b_{2}(\ell))H^{2}\cdot\operatorname{ch}_{1}^{D}(v)-b_{1}(\ell)b_{2}(\ell) H^{3}\operatorname{ch}_{0}(v).
    \]
   }
   \item{
    If $\ell$ is a numerical line for $v$ then $\ell\cap U_{v}^{D}=\ell\cap U$.
   }
   \item{
    Every wall of Lemma \ref{lemma:wallProperties} is a subset of a numerical line of $v$.
   }
   \item{
    Assume $E\in D^{b}(X)$ satisfies $[E]=v$ and $\ell$ is a numerical line of $v$.
    If $E\in\operatorname{Coh}_{b_{*}}^{D}(X)$ is $\nu_{b_{*},w_{*}}^{D}$-semistable for some $(b_{*},w_{*})\in \ell\cap U_{v}^{D}$ then $E\in\operatorname{Coh}_{b}^{D}(X)$ and $E$ is $\nu_{b,w}^{D}$-semistable for all $(b,w)\in \ell\cap U$.
   }
  \end{enumerate}
 \end{lemma}
 \begin{proof}
  The equation of $\ell$ is
  \[
   w=\tfrac{1}{2}(b_{1}(\ell)+b_{2}(\ell))b-\tfrac{1}{2}b_{1}(\ell)b_{2}(\ell),
  \]
  so (1) follows by definition of a numerical line.
  
  If $\operatorname{ch}_{0}(v)\neq 0$ then the only point of $\ell$ where $H^{2}\cdot\operatorname{ch}_{1}^{D}(v)-b H^{3}\operatorname{ch}_{0}(v)$ vanishes is $\Pi(v)$.
  Also, since $\Delta^{D}(v)\geq 0$, $\Pi(v)\not\in U$.
  Therefore, since $H^{2}\cdot\operatorname{ch}_{1}^{D}(v)-b H^{3}\operatorname{ch}_{0}(v)>0$ at some point in $\ell\cap U$, it is positive everywhere.
  If $\operatorname{ch}_{0}(v)=0$, then $H^{2}\cdot\operatorname{ch}_{1}^{D}(v)-b H^{3}\operatorname{ch}_{0}(v)$ is a positive constant.
  In either case, $\ell\cap U=\ell\cap U_{v}^{D}$.
  
  Assertion (3) follows from \cite{bayer2016}*{Proposition 12.5} (see also \cite{feyzbakhsh2023}*{Proposition 1.3}).
  
  Assertion (4) follows from \cite{bayer2016}*{Lemma 4.3}; assertion $(2)$ above excludes a wall of infinite slope.
 \end{proof}

\section{A refinement of Cheng--Feyzbakhsh's bound}
 In this section we apply the technique of \cite{bayer2016}*{Theorem 5.4} to obtain an asymptotic refinement of \cite{cheng2026}*{Corollary 1.3}.
 
 Choose $k\in\mathbb{Z}_{>0}$ such that $kH$ is very ample.
 Fix a closed immersion $\iota:X\to\mathbb{P}^{N}$ such that $\iota^{*}\mathcal{O}_{\mathbb{P}^{N}}(1)=\mathcal{O}_{X}(kH)$ and $\iota$ attains the invariant $a(X,kH)$ of \cite{cheng2026}*{Section 3}.
 Write $\mathcal{N}$ for the associated normal bundle.
 
 We restate \cite{cheng2026}*{Corollary 1.3} in our parameterization of tilt semistability.
 The recent result \cite{li2026}*{Proposition 4.5} supplies the family with real parameters required in \cite{cheng2026}*{Section 3.3}.
 Consequently, the bound of \cite{cheng2026}*{Corollary 1.3} applies to all real parameters satisfying the required positivity condition.
 
 \begin{lemma}[\cite{cheng2026}*{Corollary 1.3}]
  \label{lemma:chengFeyzbakhshInequality}
  Choose
  \[
   t_{0}>\max\left\{0,\frac{k^{2} a(X,kH)^{2}}{6}-\frac{H\cdot\operatorname{ch}_{2}(\mathcal{N})}{12 H^{3}}\right\}.
  \]
  Suppose $(b,w)\in U$ satisfies $w-\tfrac{1}{2}b^{2}\geq t_{0}$.
  If $E\in\operatorname{Coh}_{b}^{\frac{1}{2}\operatorname{ch}_{1}(\mathcal{N})}(X)$ is $\nu_{b,w}^{\frac{1}{2}\operatorname{ch}_{1}(\mathcal{N})}$-semistable with $\nu_{b,w}^{\frac{1}{2}\operatorname{ch}_{1}(\mathcal{N})}(E)=b$ then
  \[
   \operatorname{ch}_{3}^{\frac{1}{2}\operatorname{ch}_{1}(\mathcal{N})+b H}(E)+\frac{\operatorname{ch}_{2}(\mathcal{N})}{12}\cdot\operatorname{ch}_{1}^{\frac{1}{2}\operatorname{ch}_{1}(\mathcal{N})+b H}(E)<\left(3(w-\tfrac{1}{2}b^{2})+\frac{H\cdot\operatorname{ch}_{2}(\mathcal{N})}{4 H^{3}}\right)H^{2}\cdot\operatorname{ch}_{1}^{\frac{1}{2}\operatorname{ch}_{1}(\mathcal{N})+b H}(E).
  \]
 \end{lemma}
 \begin{proof}
  We write $t=w-\tfrac{1}{2}b^{2}$ and
  \[
   (\alpha,\beta)=\left(\frac{\sqrt{6t+\frac{H\cdot\operatorname{ch}_{2}(\mathcal{N})}{2 H^{3}}}}{k},\frac{b}{k}\right).
  \]
  Our choice of $t_{0}$ ensures that $\alpha>a(X,kH)$.
  In the notation of \cite{cheng2026},
  \[
   \widetilde{\operatorname{ch}}^{\beta k H}=\exp(-\beta kH)\cdot\operatorname{td}(\mathcal{N})^{-1}\cdot\operatorname{ch}=\left(1+\frac{\operatorname{ch}_{2}(\mathcal{N})}{12}\right)\cdot\operatorname{ch}^{\frac{1}{2}\operatorname{ch}_{1}(\mathcal{N})+b H}.
  \]
  In particular, the heart $\operatorname{Coh}^{B_{X}}(X)$ used there is $\operatorname{Coh}_{b}^{\frac{1}{2}\operatorname{ch}_{1}(\mathcal{N})}(X)$.
  Likewise, the tilt slope used there is
  \[
   \frac{kH\cdot\widetilde{\operatorname{ch}}_{2}^{\beta kH}(E)-\tfrac{1}{6}\alpha^{2}k^{3}H^{3}\operatorname{ch}_{0}(E)}{k^{2}H^{2}\cdot\widetilde{\operatorname{ch}}_{1}^{\beta k H}(E)}=\frac{\nu_{b,w}^{\frac{1}{2}\operatorname{ch}_{1}(\mathcal{N})}(E)-b}{k}.
  \]
  Therefore, $E$ is tilt semistable of slope zero in that parameterization.
  The claim follows from \cite{cheng2026}*{Corollary 1.3} in view of \cite{li2026}*{Proposition 4.5} and \cite{cheng2026}*{Section 3.3}.
 \end{proof}
 
 We apply a technique similar to that of \cite{bayer2016}*{Section 5} to Lemma \ref{lemma:chengFeyzbakhshInequality}.
 This improves the leading coefficient at the cost of a larger constant term.
 
 \begin{lemma}
  \label{lemma:mainTheoremText}
  Choose $t_{0}$ as in Lemma \ref{lemma:chengFeyzbakhshInequality}.
  If $(b,w)\in U$ satisfies $w-\tfrac{1}{2} b^{2}\geq t_{0}$ then every $\nu_{b,w}^{\frac{1}{2}\operatorname{ch}_{1}(\mathcal{N})}$-semistable $E\in\operatorname{Coh}_{b}^{\frac{1}{2}\operatorname{ch}_{1}(\mathcal{N})}(X)$ with $\nu_{b,w}^{\frac{1}{2}\operatorname{ch}_{1}(\mathcal{N})}(E)=b$ satisfies
  \[
   \operatorname{ch}_{3}^{\frac{1}{2}\operatorname{ch}_{1}(\mathcal{N})+b H}(E)+\frac{\operatorname{ch}_{2}(\mathcal{N})}{12}\cdot\operatorname{ch}_{1}^{\frac{1}{2}\operatorname{ch}_{1}(\mathcal{N})+b H}(E)<\left(\frac{w-\tfrac{1}{2}b^{2}}{3}+\frac{8 t_{0}}{3}+\frac{H\cdot\operatorname{ch}_{2}(\mathcal{N})}{4 H^{3}}\right)H^{2}\cdot\operatorname{ch}_{1}^{\frac{1}{2}\operatorname{ch}_{1}(\mathcal{N})+b H}(E).
  \]
 \end{lemma}
 \begin{proof}
  For each $E\in D^{b}(X)$ and $(b,w)\in U$, define
  \begin{align*}
   \nabla_{b,w}^{\frac{1}{2}\operatorname{ch}_{1}(\mathcal{N})}(E) & =\operatorname{ch}_{3}^{\frac{1}{2}\operatorname{ch}_{1}(\mathcal{N})+b H}(E)+\frac{\operatorname{ch}_{2}(\mathcal{N})}{12}\cdot\operatorname{ch}_{1}^{\frac{1}{2}\operatorname{ch}_{1}(\mathcal{N})+b H}(E) \\
   & \qquad-\left(\frac{w-\tfrac{1}{2}b^{2}}{3}+\frac{8 t_{0}}{3}+\frac{H\cdot\operatorname{ch}_{2}(\mathcal{N})}{4 H^{3}}\right)H^{2}\cdot\operatorname{ch}_{1}^{\frac{1}{2}\operatorname{ch}_{1}(\mathcal{N})+b H}(E).
  \end{align*}
  For fixed $(b,w)\in U$, $\nabla_{b,w}^{\frac{1}{2}\operatorname{ch}_{1}(\mathcal{N})}$ is additive in short exact sequences.
  Our goal is to prove $\nabla_{b,w}^{\frac{1}{2}\operatorname{ch}_{1}(\mathcal{N})}(E)<0$ under the hypotheses of the lemma.

  We proceed by contradiction.
  Since $H,\operatorname{ch}_{1}(\mathcal{N})\in\operatorname{NS}(X)$, $4\Delta^{\frac{1}{2}\operatorname{ch}_{1}(\mathcal{N})}(E)\in\mathbb{Z}$ for all $E\in D^{b}(X)$.
  Moreover, by Lemma \ref{lemma:basicFactsTiltStability}, $\Delta^{\frac{1}{2}\operatorname{ch}_{1}(\mathcal{N})}(E)\geq 0$ for all $E$ satisfying the hypotheses of the lemma, so $\Delta^{\frac{1}{2}\operatorname{ch}_{1}(\mathcal{N})}$ takes values in the discrete set $\tfrac{1}{4}\mathbb{Z}_{\geq 0}$.
  Therefore, if the lemma fails then we may choose a counterexample with minimal discriminant.
  Explicitly, we may choose $(b_{1},w_{1})\in U$ with $w_{1}-\tfrac{1}{2}b_{1}^{2}\geq t_{0}$ and a $\nu_{b_{1},w_{1}}^{\frac{1}{2}\operatorname{ch}_{1}(\mathcal{N})}$-semistable $E\in\operatorname{Coh}_{b_{1}}^{\frac{1}{2}\operatorname{ch}_{1}(\mathcal{N})}(X)$ satisfying $\nu_{b_{1},w_{1}}^{\frac{1}{2}\operatorname{ch}_{1}(\mathcal{N})}(E)=b_{1}$, $\nabla_{b_{1},w_{1}}^{\frac{1}{2}\operatorname{ch}_{1}(\mathcal{N})}(E)\geq 0$, and $\Delta^{\frac{1}{2}\operatorname{ch}_{1}(\mathcal{N})}(E)$ minimal among all such counterexamples.
  For brevity, write $t_{1}=w_{1}-\frac{1}{2}b_{1}^{2}$.
  
  Our argument decreases $w-\tfrac{1}{2}b^{2}$ while preserving $\nu_{b,w}^{\frac{1}{2}\operatorname{ch}_{1}(\mathcal{N})}(E)=b$.
  To this end, we consider the smooth function $b:[t_{0},t_{1}]\to\mathbb{R}$ uniquely determined by the equations
  \begin{gather*}
   H\cdot\operatorname{ch}^{\frac{1}{2}\operatorname{ch}_{1}(\mathcal{N})+b(t)H}_{2}(E)=t H^{3}\operatorname{ch}_{0}(E) \\
   H^{2}\cdot\operatorname{ch}_{1}^{\frac{1}{2}\operatorname{ch}_{1}(\mathcal{N})+b(t)H}(E)>0.
  \end{gather*}
  Explicitly,
  \[
   b(t)=
   \begin{cases}
    \frac{H^{2}\cdot\operatorname{ch}_{1}^{\frac{1}{2}\operatorname{ch}_{1}(\mathcal{N})}(E)-\sqrt{\Delta^{\frac{1}{2}\operatorname{ch}_{1}(\mathcal{N})}(E)+2 t(H^{3}\operatorname{ch}_{0}(E))^{2}}}{H^{3}\operatorname{ch}_{0}(E)}: & \operatorname{ch}_{0}(E)\neq 0 \\
    \frac{H\cdot\operatorname{ch}_{2}^{\frac{1}{2}\operatorname{ch}_{1}(\mathcal{N})}(E)}{H^{2}\cdot\operatorname{ch}_{1}^{\frac{1}{2}\operatorname{ch}_{1}(\mathcal{N})}(E)}: & \operatorname{ch}_{0}(E)=0
   \end{cases}.
  \]
  In the first case the quantity inside the square root is positive because $\Delta^{\frac{1}{2}\operatorname{ch}_{1}(\mathcal{N})}(E)\geq 0$ and $t\geq t_{0}>0$.
  We also set $w(t)=t+\tfrac{1}{2}b(t)^{2}$.
  By construction, $(b(t_{1}),w(t_{1}))=(b_{1},w_{1})$ and $\nu_{b(t),w(t)}^{\frac{1}{2}\operatorname{ch}_{1}(\mathcal{N})}(E)=b(t)$.
  
  Since $\nu_{b(t),w(t)}^{\frac{1}{2}\operatorname{ch}_{1}(\mathcal{N})}(E)=b(t)$, we find
  \[
   H^{2}\cdot\operatorname{ch}_{1}^{\frac{1}{2}\operatorname{ch}_{1}(\mathcal{N})+b(t)H}(E)=\sqrt{\Delta^{\frac{1}{2}\operatorname{ch}_{1}(\mathcal{N})}(E)+2 t (H^{3}\operatorname{ch}_{0}(E))^{2}}.
  \]
  It follows that $H^{2}\cdot\operatorname{ch}_{1}^{\frac{1}{2}\operatorname{ch}_{1}(\mathcal{N})+b(t) H}(E)>0$ for all $t\in [t_{0},t_{1}]$.

  Differentiating the defining equation for $b(t)$ gives
  \[
   -H^{2}\cdot\operatorname{ch}_{1}^{\frac{1}{2}\operatorname{ch}_{1}(\mathcal{N})+b(t)H}(E) \frac{d}{dt} b(t)=H^{3}\operatorname{ch}_{0}(E)
  \]
  and so
  \[
   \frac{d}{dt} b(t)=-\frac{H^{3}\operatorname{ch}_{0}(E)}{H^{2}\cdot\operatorname{ch}_{1}^{\frac{1}{2}\operatorname{ch}_{1}(\mathcal{N})+b(t)H}(E)}.
  \]
  Hence
  \begin{align*}
   \frac{d}{dt} \nabla_{b(t),w(t)}^{\frac{1}{2}\operatorname{ch}_{1}(\mathcal{N})}(E) & \stackrel{(1)}{=}-\frac{\Delta^{\frac{1}{2}\operatorname{ch}_{1}(\mathcal{N})}(E)+\left(8t_{0}+\frac{H\cdot\operatorname{ch}_{2}(\mathcal{N})}{2 H^{3}}\right)(H^{3}\operatorname{ch}_{0}(E))^{2}}{3 H^{2}\cdot\operatorname{ch}_{1}^{\frac{1}{2}\operatorname{ch}_{1}(\mathcal{N})+b(t)H}(E)} \\
    & \stackrel{(2)}{\leq} -\frac{\left(8 t_{0}+\frac{H\cdot\operatorname{ch}_{2}(\mathcal{N})}{2 H^{3}}\right)\left(H^{3}\operatorname{ch}_{0}(E)\right)^{2}}{3 H^{2}\cdot\operatorname{ch}_{1}^{\frac{1}{2}\operatorname{ch}_{1}(\mathcal{N})+b(t)H}(E)} \\
    & \stackrel{(3)}{\leq} 0.
  \end{align*}
  The numbered steps are justified as follows.
  \begin{enumerate}
   \item{
    Direct calculation,
   }
   \item{
    $\Delta^{\frac{1}{2}\operatorname{ch}_{1}(\mathcal{N})}(E)\geq 0$ and $H^{2}\cdot\operatorname{ch}_{1}^{\frac{1}{2}\operatorname{ch}_{1}(\mathcal{N})+b(t)H}(E)>0$, and
   }
   \item{
    Since $t_{0}+\frac{H\cdot\operatorname{ch}_{2}(\mathcal{N})}{12 H^{3}}>\frac{k^{2} a(X,kH)^{2}}{6}\geq 0$ and $t_{0}>0$, $8t_{0}+\frac{H\cdot\operatorname{ch}_{2}(\mathcal{N})}{2 H^{3}}=6\left(t_{0}+\frac{H\cdot\operatorname{ch}_{2}(\mathcal{N})}{12 H^{3}}\right)+2t_{0}>0$.
   }
  \end{enumerate}
  Thus, $t\mapsto\nabla_{b(t),w(t)}^{\frac{1}{2}\operatorname{ch}_{1}(\mathcal{N})}$ is nonincreasing on $[t_{0},t_{1}]$.
  
  Consider the set
  \[
   S_{E}=\{t\in [t_{0},t_{1}]:E\in\operatorname{Coh}_{b(t)}^{\frac{1}{2}\operatorname{ch}_{1}(\mathcal{N})}(X), E \ \text{is} \ \nu_{b(t),w(t)}^{\frac{1}{2}\operatorname{ch}_{1}(\mathcal{N})}-\text{semistable}\}.
  \]
  Since $H^{2}\cdot\operatorname{ch}_{1}^{\frac{1}{2}\operatorname{ch}_{1}(\mathcal{N})+b(t) H}(E)>0$ for all $t\in [t_{0},t_{1}]$, $t\mapsto (b(t),w(t))$ is a continuous function $[t_{0},t_{1}]\to U_{[E]}^{\frac{1}{2}\operatorname{ch}_{1}(\mathcal{N})}$.
  By Lemma \ref{lemma:wallProperties}, $U_{E}^{\frac{1}{2}\operatorname{ch}_{1}(\mathcal{N}),\mathrm{ss}}$ is closed in $U_{[E]}^{\frac{1}{2}\operatorname{ch}_{1}(\mathcal{N})}$ and so $S_{E}$ is closed in $[t_{0},t_{1}]$.
  Since $t_{1}\in S_{E}$, $S_{E}$ is nonempty and so $t_{*}=\min S_{E}$ exists.
  Since $t\mapsto\nabla_{b(t),w(t)}^{\frac{1}{2}\operatorname{ch}_{1}(\mathcal{N})}$ is nonincreasing on $[t_{0},t_{1}]$, $0\leq \nabla_{b_{1},w_{1}}^{\frac{1}{2}\operatorname{ch}_{1}(\mathcal{N})}(E)\leq \nabla_{b(t_{*}),w(t_{*})}^{\frac{1}{2}\operatorname{ch}_{1}(\mathcal{N})}(E)$.
  
  If $t_{*}=t_{0}$ then $\tfrac{1}{3}t_{*}+\tfrac{8}{3}t_{0}=3t_{*}$, so, by Lemma \ref{lemma:chengFeyzbakhshInequality}, $\nabla^{\frac{1}{2}\operatorname{ch}_{1}(\mathcal{N})}_{b(t_{0}),w(t_{0})}(E)<0$ which is a contradiction.
  If $t_{*}>t_{0}$ then $(b(t_{*}),w(t_{*}))\in \partial U^{\frac{1}{2}\operatorname{ch}_{1}(\mathcal{N}),\mathrm{ss}}_{E}$.
  Therefore, by Lemma \ref{lemma:wallProperties}, there exists a short exact sequence
  \[
   0\to F\to E\to G\to 0
  \]
  in $\operatorname{Coh}_{b(t_{*})}^{\frac{1}{2}\operatorname{ch}_{1}(\mathcal{N})}(X)$ where $F,G$ are nonzero $\nu_{b(t_{*}),w(t_{*})}^{\frac{1}{2}\operatorname{ch}_{1}(\mathcal{N})}$-semistable objects with $\nu_{b(t_{*}),w(t_{*})}^{\frac{1}{2}\operatorname{ch}_{1}(\mathcal{N})}(F)=\nu_{b(t_{*}),w(t_{*})}^{\frac{1}{2}\operatorname{ch}_{1}(\mathcal{N})}(E)=b(t_{*})$,$\nu_{b(t_{*}),w(t_{*})}^{\frac{1}{2}\operatorname{ch}_{1}(\mathcal{N})}(G)=\nu_{b(t_{*}),w(t_{*})}^{\frac{1}{2}\operatorname{ch}_{1}(\mathcal{N})}(E)=b(t_{*})$,  and
  \[
   0\leq\Delta^{\frac{1}{2}\operatorname{ch}_{1}(\mathcal{N})}(F),\Delta^{\frac{1}{2}\operatorname{ch}_{1}(\mathcal{N})}(G)<\Delta^{\frac{1}{2}\operatorname{ch}_{1}(\mathcal{N})}(E).
  \]
  Since $E$ is a counterexample with minimal $\Delta^{\frac{1}{2}\operatorname{ch}_{1}(\mathcal{N})}$, the conclusion of the lemma holds for $F$ and $G$.
  Therefore, by additivity, we find
  \[
   \nabla_{b(t_{*}),w(t_{*})}^{\frac{1}{2}\operatorname{ch}_{1}(\mathcal{N})}(E)=\nabla_{b(t_{*}),w(t_{*})}^{\frac{1}{2}\operatorname{ch}_{1}(\mathcal{N})}(F)+\nabla_{b(t_{*}),w(t_{*})}^{\frac{1}{2}\operatorname{ch}_{1}(\mathcal{N})}(G)<0
  \]
  a contradiction.
 \end{proof}
 
 \begin{theorem}
  \label{theorem:asymptoticGammaBMT}
  Let $(X,H)$ be a smooth polarized complex threefold.
  Set $D=\tfrac{1}{2}K_{X}$.
  There exist $t_{0}>0$ and $\Gamma\in N^{2}(X)_{\mathbb{Q}}$ such that every $\nu_{b,w}^{D}$-semistable $E\in\operatorname{Coh}_{b}^{D}(X)$ with $w-\tfrac{1}{2}b^{2}\geq t_{0}$ and $\nu_{b,w}^{D}(E)=b$ satisfies
  \[
   \operatorname{ch}_{3}^{D+b H}(E)\leq
  \frac{w-\frac{1}{2}b^{2}}{3}H^{2}\cdot\operatorname{ch}_{1}^{D+bH}(E)+\Gamma\cdot\operatorname{ch}_{1}^{D+bH}(E).
   \]
 \end{theorem}
 \begin{proof}
  Choose $t_{0}$ as in Lemma \ref{lemma:chengFeyzbakhshInequality} and increase it, if necessary, so that $t_{0}\in\mathbb{Q}$.
  Set
  \[
   \Gamma=\left(\frac{8 t_{0}}{3}+\frac{H\cdot\operatorname{ch}_{2}(\mathcal{N})}{4 H^{3}}\right)H^{2}-\frac{\operatorname{ch}_{2}(\mathcal{N})}{12}\in N^{2}(X)_{\mathbb{Q}}
  \]
  Since $\operatorname{ch}_{1}(\mathcal{N})=K_{X}+(N+1)kH$, the claim follows from Lemma \ref{lemma:mainTheoremText} after replacing $b$ by $b-\tfrac{1}{2}(N+1)k$ and keeping $w-\tfrac{1}{2}b^{2}$ fixed.
 \end{proof}

\section{The wBG condition for certain polarized threefolds}
 In this section we apply Theorem \ref{theorem:asymptoticGammaBMT} to show that any smooth polarized complex threefold with $K_{X}\equiv\lambda H$ for some $\lambda\in\mathbb{Q}$ satisfies $(\mathbf{wBG})$ of \cite{feyzbakhsh2024}.

 We assume $(X,H)$ is a smooth polarized complex threefold.
 From now on, we also assume $D=0$.
 Therefore, we omit the superscript $0$ from our notation, writing $\mu_{b}$, $\nu_{b,w}$, $U_{v}$, and $\operatorname{Coh}_{b}(X)$ rather than $\mu_{b}^{0}$, $\nu_{b,w}^{0}$ , $U_{v}^{0}$, and $\operatorname{Coh}_{b}^{0}(X)$ respectively.
 
\subsection{The wBG condition} 
 Fix $v\in K_{\mathrm{num}}(X)$ with $\operatorname{ch}_{0}(v)>0$ and write $v_{n}=v-[\mathcal{O}_{X}(-nH)]$.
 For all $n\gg 0$, we write $\ell_{v,n}$ for the numerical line of $v_{n}$ defined in \cite{feyzbakhsh2024}*{Section 3.1}.
 Explicitly, for $n\gg 0$ consider all numerical lines $\ell$ for $v_{n}$ satisfying
 \begin{gather*}
  b_{1}(\ell)\leq -n+\epsilon \\
  \mu_{0}(v)-\epsilon\leq b_{2}(\ell)<\mu_{0}(v).
 \end{gather*}
 where $\epsilon=1/(4H^{3}\operatorname{ch}_{0}(v)^{2})$.
 For all $n\gg 0$, we define $\ell_{v,n}$ to be the numerical line satisfying these inequalities with maximal $b_{1}(\ell)$.
 
 We record a lower bound on $b_{1}(\ell_{v,n})$ that will be useful later.
 Indeed, for $n\gg 0$ consider a numerical line, $\ell_{n}$, with $b_{1}(\ell_{n})=-n+\tfrac{1}{2}\epsilon$.
 By Lemma \ref{lemma:factsAboutNumericalLines}(1), we find
 \[
  \lim_{n\to\infty}b_{2}(\ell_{n})=\mu_{0}(v)-\frac{\epsilon}{2\operatorname{ch}_{0}(v)}.
 \]
 In particular, $\mu_{0}(v)-\epsilon<b_{2}(\ell_{n})<\mu_{0}(v)$ for all $n\gg 0$.

 The maximal choice of $\ell_{v,n}$ therefore gives
 \begin{equation}
  \label{equation:lowerBoundLeftEndpoint}
  -n+\frac{\epsilon}{2}\leq b_{1}(\ell_{v,n}).
 \end{equation}
 
 \begin{definition}
  \label{definition:wBG}
  We say that $(X,H)$ satisfies $(\mathbf{wBG})$ if for every $v\in K_{\mathrm{num}}(X)$ with $\operatorname{ch}_{0}(v)\geq 1$, all $n\gg 0$ (depending on $v$), and all $(b,w)\in \ell_{v,n}\cap U$ there is no $\nu_{b,w}$-semistable object $E\in\operatorname{Coh}_{b}(X)$ with $[E]=v_{n}$.
 \end{definition}
 
 By Lemma \ref{lemma:factsAboutNumericalLines}(2), $\ell_{v,n}\cap U_{v_{n}}=\ell_{v,n}\cap U$.
 We will tacitly use this equality when discussing $(\mathbf{wBG})$.

\subsection{Verifying the wBG condition}
 Assume $K_{X}\equiv\lambda H$ for some $\lambda\in\mathbb{Q}$.
 We show that $(X,H)$ satisfies $(\mathbf{wBG})$ under this assumption.

 Recall that Lemma \ref{lemma:mainTheoremText} is proved for $D=\frac{1}{2}\operatorname{ch}_{1}(\mathcal{N})$ rather than $D=0$.
 The following result translates Lemma \ref{lemma:mainTheoremText} to the $D=0$ case.
 
 \begin{corollary}
  \label{bgInequalityProportionalH}
  Assume $(X,H)$ is a smooth polarized threefold with $K_{X}\equiv\lambda H$ for some $\lambda\in\mathbb{Q}$.
  Choose $t_{0}$ as in Lemma \ref{lemma:chengFeyzbakhshInequality}.
  There exists a constant $M>0$, depending only on $(X,H)$, such that if $(b,w)\in U$ satisfies $w-\tfrac{1}{2} b^{2}\geq t_{0}$ then every $\nu_{b,w}$-semistable $E\in\operatorname{Coh}_{b}(X)$ with $\nu_{b,w}(E)=b$ satisfies
  \[
   \operatorname{ch}_{3}^{b H}(E)\leq\left(\frac{w-\tfrac{1}{2}b^{2}}{3}+M\right) H^{2}\cdot\operatorname{ch}_{1}^{b H}(E).
  \]
 \end{corollary}
 \begin{proof}
  Since $\tfrac{1}{2}K_{X}\equiv\tfrac{1}{2}\lambda H$, applying Theorem \ref{theorem:asymptoticGammaBMT} at $(b-\tfrac{1}{2}\lambda,w-\tfrac{1}{2}\lambda b+\tfrac{1}{8}\lambda^{2})$ shows
  \[
   \operatorname{ch}_{3}^{b H}(E)\leq\frac{w-\frac{1}{2}b^{2}}{3}H^{2}\cdot\operatorname{ch}_{1}^{b H}(E)+\Gamma\cdot\operatorname{ch}_{1}^{b H}(E).
  \]
  By \cite{bayer2016}*{Lemma 3.3 and Theorem 3.5}, there exists a constant $C_{H}\geq 0$, depending only on $H$, such that
  \[
   H\cdot\left(\operatorname{ch}_{1}^{b H}(E)\right)^{2}\geq - C_{H}\left(H^{2}\cdot\operatorname{ch}_{1}^{b H}(E)\right)^{2}.
  \]
  On the other hand, by the Hodge index theorem,
  \[
   \{B\in N^{1}(X)_{\mathbb{R}}:H^{2}\cdot B=1,H\cdot B^{2}\geq -C_{H}\}
  \]
  is compact.
  By the preceding inequality, $\operatorname{ch}_{1}^{b H}(E)/(H^{2}\cdot\operatorname{ch}_{1}^{bH}(E))$ lies in this compact set.
  Therefore, there exists $M>0$, depending only on $(X,H)$, such that
  \[
   \Gamma\cdot\operatorname{ch}_{1}^{b H}(E)\leq M H^{2}\cdot\operatorname{ch}_{1}^{b H}(E).
  \]
  Substituting this bound into our first inequality gives the claim.
 \end{proof}
 
 We translate Corollary \ref{bgInequalityProportionalH} into an obstruction to certain numerical lines (see Definition \ref{definition:numericalLine}).
 
 \begin{lemma}
  \label{lemma:numericalWallBound}
  Assume that $\ell$ is a numerical line for $v$ and that there exists $\nu_{b,w}$-semistable $E\in\operatorname{Coh}_{b}(X)$ with $[E]=v$ for some $(b,w)\in\ell\cap U$.
  Choose $t_{0}$ as in Lemma \ref{lemma:chengFeyzbakhshInequality}.
  If $(b_{2}(\ell)-b_{1}(\ell))^{2}\geq 8 t_{0}$ then
  \[
   \operatorname{ch}_{3}(E)-\frac{b_{1}(\ell)+b_{2}(\ell)}{3}H\cdot\operatorname{ch}_{2}(E)+\frac{b_{1}(\ell)b_{2}(\ell)}{6}H^{2}\cdot\operatorname{ch}_{1}(E)\leq M\left(H^{2}\cdot\operatorname{ch}_{1}(E)-\frac{b_{1}(\ell)+b_{2}(\ell)}{2}H^{3}\operatorname{ch}_{0}(E)\right).
  \]
  where $M$ is the constant of Corollary \ref{bgInequalityProportionalH}.
 \end{lemma}
 \begin{proof}
  We write $b_{*}=\tfrac{1}{2}(b_{1}(\ell)+b_{2}(\ell))$ and $w_{*}=\tfrac{1}{4}(b_{1}(\ell)^{2}+b_{2}(\ell)^{2})$.
  By Lemma \ref{lemma:factsAboutNumericalLines}, $E\in\operatorname{Coh}_{b_{*}}(X)$ and is $\nu_{b_{*},w_{*}}$-semistable.
  By choice of $b_{*}$ and $w_{*}$, we have $\nu_{b_{*},w_{*}}(E)=b_{*}$.
  Moreover, by assumption, $w_{*}-\tfrac{1}{2}b_{*}^{2}=\tfrac{1}{8}(b_{2}(\ell)-b_{1}(\ell))^{2}\geq t_{0}$.
  Therefore, applying Corollary \ref{bgInequalityProportionalH} to $E$ at $(b_{*},w_{*})$ gives
  \[
   \operatorname{ch}_{3}^{b_{*}H}(E)\leq\left(\frac{(b_{2}(\ell)-b_{1}(\ell))^{2}}{24}+M\right)H^{2}\cdot\operatorname{ch}_{1}^{b_{*}H}(E).
  \]
  Rearranging this inequality and applying Lemma \ref{lemma:factsAboutNumericalLines}(1) shows
  \begin{align*}
   \operatorname{ch}_{3}(E)-\frac{b_{1}(\ell)+b_{2}(\ell)}{3}H\cdot\operatorname{ch}_{2}(E)+\frac{b_{1}(\ell)b_{2}(\ell)}{6}H^{2}\cdot\operatorname{ch}_{1}(E) & = \operatorname{ch}_{3}^{b_{*}H}(E)-\frac{(b_{2}(\ell)-b_{1}(\ell))^{2}}{24} H^{2}\cdot\operatorname{ch}_{1}^{b_{*}H}(E) \\
     & \leq M H^{2}\cdot\operatorname{ch}_{1}^{b_{*}H}(E),
  \end{align*} 
  as desired.
 \end{proof}
 
 We now show that if $K_{X}\equiv\lambda H$ for some $\lambda\in\mathbb{Q}$ then $(X,H)$ satisfies $(\mathbf{wBG})$.

 \begin{proposition}
  \label{proposition:threefoldSatisfiesWBG}
  If $(X,H)$ is a smooth polarized complex threefold with $K_{X}\equiv\lambda H$ for some $\lambda\in\mathbb{Q}$ then $(X,H)$ satisfies $(\mathbf{wBG})$.
 \end{proposition}
 \begin{proof}
  We proceed by contradiction.
  Fix $v\in K_{\mathrm{num}}(X)$ with $\operatorname{ch}_{0}(v)\geq 1$.
  Suppose that, for arbitrarily large $n$ there exists a $\nu_{b,w}$-semistable $E\in\operatorname{Coh}_{b}(X)$ with $[E]=v_{n}$ for some $(b,w)\in\ell_{v,n}\cap U$.
  Since $\mu_{0}(v)-\epsilon\leq b_{2}(\ell_{v,n})$ and $b_{1}(\ell_{v,n})\leq-n+\epsilon$, $\ell_{v,n}$ satisfies the assumptions of Lemma \ref{lemma:numericalWallBound} for all $n\gg 0$.
  Applying Lemma \ref{lemma:numericalWallBound} and rewriting $\operatorname{ch}(v_{n})$ in terms of $\operatorname{ch}(v)$ shows
  \begin{align*}
   \frac{H^{3}}{6}n (n+b_{1}(\ell_{v,n}))(n+b_{2}(\ell_{v,n}))&\leq M\left(H^{2}\cdot\operatorname{ch}_{1}(v)+n H^{3}-\frac{b_{1}(\ell_{v,n})+b_{2}(\ell_{v,n})}{2}H^{3}(\operatorname{ch}_{0}(v)-1)\right) \\
   &\quad-\left(\operatorname{ch}_{3}(v)-\frac{b_{1}(\ell_{v,n})+b_{2}(\ell_{v,n})}{3}H\cdot\operatorname{ch}_{2}(v)+\frac{b_{1}(\ell_{v,n})b_{2}(\ell_{v,n})}{6}H^{2}\cdot\operatorname{ch}_{1}(v)\right)
  \end{align*}
  By definition of $\ell_{v,n}$ and Equation \ref{equation:lowerBoundLeftEndpoint},
  \[
   (b_{1}(\ell_{v,n}),b_{2}(\ell_{v,n}))\in\left[\tfrac{1}{2}\epsilon-n,\epsilon-n\right]\times\left[\mu_{0}(v)-\epsilon,\mu_{0}(v)\right]
  \]
  These bounds allow us to choose $C>0$, depending only on $v$ and $M$, such that
  \begin{align*}
   &M\left(H^{2}\cdot\operatorname{ch}_{1}(v)+n H^{3}-\frac{b_{1}(\ell_{v,n})+b_{2}(\ell_{v,n})}{2}H^{3}(\operatorname{ch}_{0}(v)-1)\right) \\
   \qquad &-\left(\operatorname{ch}_{3}(v)-\frac{b_{1}(\ell_{v,n})+b_{2}(\ell_{v,n})}{3}H\cdot\operatorname{ch}_{2}(v)+\frac{b_{1}(\ell_{v,n})b_{2}(\ell_{v,n})}{6}H^{2}\cdot\operatorname{ch}_{1}(v)\right) \\
   &\qquad \leq C(1+n)
  \end{align*}
  Thus
  \[
   \frac{H^{3}}{6} n(n+b_{1}(\ell_{v,n}))(n+b_{2}(\ell_{v,n}))\leq C(1+n)
  \]
  where $C$ depends only on $v$ and $M$.

  On the other hand, $n+b_{2}(\ell_{v,n})\geq \tfrac{1}{2}n$ for all $n\gg 0$.
  Likewise, by Equation \ref{equation:lowerBoundLeftEndpoint}, $n+b_{1}(\ell_{v,n})\geq\tfrac{1}{2}\epsilon$.
  Therefore, for all $n\gg 0$,
  \[
   \frac{H^{3}}{6}n(n+b_{1}(\ell_{v,n}))(n+b_{2}(\ell_{v,n}))\geq\frac{H^{3}}{24}n^{2}\epsilon.
  \]
  Combining these inequalities gives
  \[
   \frac{H^{3}}{24} n^{2}\epsilon\leq C(1+n).
  \]
  This is a contradiction for $n\gg 0$.
 \end{proof}

\section{Application to Donaldson--Thomas theory}
 In this section, we show that the generalized Donaldson--Thomas (DT) invariants counting Gieseker semistable coherent sheaves can be expressed in terms of rank zero invariants, and then in terms of rank one invariants.
 Thus, both \cite{feyzbakhsh2024}*{Theorem 1} and \cite{feyzbakhsh2023}*{Theorem 1} hold unconditionally.
 
 We briefly recall generalized Donaldson-Thomas invariants.
 Assume $X$ is a smooth projective threefold with $K_{X}\cong\mathcal{O}_{X}$ and $H^{1}(X,\mathcal{O}_{X})=0$. 
 Fix $v\in K_{\mathrm{num}}(X)$ and $(b,w)\in U_{v}$.
 We write $\mathsf{J}(v)$ for the generalized Donaldson-Thomas invariant of Gieseker semistable sheaves (with respect to $H$) in $\operatorname{Coh}(X)$~\cite{joyce2012}.
 We write $\mathsf{J}_{b,w}(v)\in\mathbb{Q}$ for the generalized Donaldson--Thomas invariant of $\nu_{b,w}$-semistable objects of numerical class $v$, constructed as in \cite{feyzbakhsh2024}*{Section 4}.
 By Lemma \ref{lemma:wallProperties}(4), the limit $\mathsf{J}_{b,\infty}(v)=\lim_{w\to\infty} \mathsf{J}_{b,w}(v)$ is well-defined.
 
 The construction of $\mathsf{J}_{b,w}(v)$ in \cite{feyzbakhsh2024}*{Section 4} uses the Bayer--Macr{\`i}--Toda (BMT) conjecture.
 Specifically, the proof of \cite{feyzbakhsh2024}*{Proposition 4.1(i)} establishes that the stack of $\nu_{b,w}$-semistable objects of numerical class $v$ is algebraic of finite type by invoking \cite{feyzbakhsh2024}*{Theorem C.2 and Theorem C.5} which assume the BMT conjecture.
 This dependence can be avoided: openness and boundedness follow from \cite{bayer2021}*{Lemma 21.12 and the proof of Proposition 25.3} and \cite{piyaratne2019}*{Corollary 4.18} respectively.
 With this modification, the construction of $\mathsf{J}_{b,w}(v)$ in \cite{feyzbakhsh2024}*{Section 4} goes through unchanged.
 In particular, Feyzbakhsh--Thomas's analogue of the Joyce--Song wall-crossing formula holds.
 
 Proposition \ref{proposition:threefoldSatisfiesWBG} gives the reduction to rank zero invariants.
 
 \begin{corollary}[cf. \cite{feyzbakhsh2024}*{Theorem 1}]
  \label{corollary:rankZeroReduction}
  Assume $(X,H)$ is a polarized Calabi--Yau threefold.
  Fix $v\in K_{\mathrm{num}}(X)$ with $\operatorname{ch}_{0}(v)\geq 1$.
  Then $\mathsf{J}(v)$ can be written as a polynomial with rational coefficients (depending on $v$, $(X,H)$, and the graded ring $H^{\bullet}(X,\mathbb{Q})$) in finitely many invariants
  \[
   \mathsf{J}(\alpha_{1}),\ldots,\mathsf{J}(\alpha_{m})
  \]
  where $\operatorname{ch}_{0}(\alpha_{i})=0$ and $H^{2}\cdot\operatorname{ch}_{1}(\alpha_{i})>0$.
 \end{corollary}
 \begin{proof}
  By Proposition \ref{proposition:threefoldSatisfiesWBG}, $(X,H)$ satisfies $(\mathbf{wBG})$.
  The result follows from \cite{feyzbakhsh2024}*{Theorem 1}.
 \end{proof}
 
 It remains to reduce rank zero invariants to rank one invariants.
 
\subsection{Numerical estimates}
 In this subsection we give some bounds on $\operatorname{ch}_{3}$ for objects of interest.

 We use big-$O$ notation.
 For fixed $v\in K_{\mathrm{num}}(X)$, we write $f(n)=O(n^{a})$ if $|f(n)|\leq C n^{a}$ for all sufficiently large $n$ where $C>0$ and the lower bound on $n$ depend only on $v$ and $(X,H)$.
 We state any different dependence of the implied constant explicitly.
 
 The following lemma controls destabilizing objects on sufficiently wide numerical lines.

 \begin{lemma}[cf. \cite{feyzbakhsh2023}*{Lemmas 2.1 and 2.2}]
  \label{lemma:filtrationExists}
  Fix $(b,w)\in U$.
  Let $E\in\operatorname{Coh}_{b}(X)$ be a $\nu_{b,w}$-semistable object with $\operatorname{ch}_{0}(E)=-1$ and $\nu_{b,w}(E)<+\infty$.
  Suppose $\ell$ is a numerical line for $[E]$ through $(b,w)$ satisfying
  \[
   H^{2}\cdot\operatorname{ch}_{1}(E)+b_{1}(\ell)H^{3}<(b_{2}(\ell)-b_{1}(\ell))H^{3}.
  \]
  Assume
  \[
   0=E_{0}\subseteq E_{1}\subseteq\cdots\subseteq E_{m}=E
  \]
  is a filtration in $\operatorname{Coh}_{b}(X)$ with nonzero $\nu_{b,w}$-semistable consecutive factors $Q_{i}=E_{i}/E_{i-1}$ satisfying $\nu_{b,w}(Q_{i})=\nu_{b,w}(E)$.
  Then there exists a unique $i_{0}$ such that $\operatorname{ch}_{0}(Q_{i_{0}})=-1$.
  Moreover, the following statements hold:
  \begin{enumerate}
   \item{
    $-b_{1}(\ell)H^{3}\leq H^{2}\cdot\operatorname{ch}_{1}(Q_{i_{0}})$,
   }
   \item{
    $H^{2}\cdot\operatorname{ch}_{1}(Q_{i_{0}})\leq H^{2}\cdot\operatorname{ch}_{1}(E)$ with equality if and only if $m=1$,
   }
   \item{
    If $i\neq i_{0}$ then $Q_{i}$ is a pure sheaf supported in codimension one,
   }
   \item{
    If $i\neq i_{0}$ then no destabilizing wall of $[Q_{i}]$ intersects the closed ray $\{(b,w^{\prime}):w^{\prime}\geq w\}$.
   }
  \end{enumerate}
 \end{lemma}
 \begin{proof}
  By assumption, $\ell$ is a numerical line of $[Q_{i}]$.
  Therefore, by Lemma \ref{lemma:factsAboutNumericalLines}(4) $Q_{i}\in\operatorname{Coh}_{b^{\prime}}(X)$ for every $b^{\prime}\in(b_{1}(\ell),b_{2}(\ell))$.
  Taking the limit as $b^{\prime}$ approaches the endpoints of this open interval shows
  \begin{gather*}
   H^{2}\cdot\operatorname{ch}_{1}(\mathcal{H}^{-1}(Q_{i}))\leq b_{1}(\ell)H^{3}\operatorname{ch}_{0}(\mathcal{H}^{-1}(Q_{i})) \\
   H^{2}\cdot\operatorname{ch}_{1}(\mathcal{H}^{0}(Q_{i}))\geq b_{2}(\ell)H^{3}\operatorname{ch}_{0}(\mathcal{H}^{0}(Q_{i}))
  \end{gather*}
  where $\mathcal{H}^{i}:D^{b}(X)\to\operatorname{Coh}(X)$ denotes  cohomology with respect to the standard $t$-structure.
  Summing over all $Q_{i}$ gives
  \[
   (b_{2}(\ell)-b_{1}(\ell))H^{3}\sum_{i}\operatorname{ch}_{0}(\mathcal{H}^{0}(Q_{i}))\leq H^{2}\cdot\operatorname{ch}_{1}(E)+b_{1}(\ell)H^{3}.
  \]
  The assumed bound on $(b_{2}(\ell)-b_{1}(\ell))H^{3}$ implies that $\operatorname{ch}_{0}(\mathcal{H}^{0}(Q_{i}))=0$ for all $i$.
  Since $\sum_{i}\operatorname{ch}_{0}(Q_{i})=-1$, there exists a unique $i_{0}$ such that $\operatorname{ch}_{0}(Q_{i_{0}})=-1$ and $\operatorname{ch}_{0}(Q_{i})=0$ for all $i\neq i_{0}$.
  For $i\neq i_{0}$, it follows that $Q_i$ is a sheaf of rank zero.
  Since $Q_{i}$ is $\nu_{b,w}$-semistable with $\nu_{b,w}(Q_{i})<+\infty$, it is pure of codimension one.
  
  Since $\mathcal H^{0}(Q_{i_0})$ is torsion and $\mathcal H^{-1}(Q_{i_0})$ has rank one, the above inequalities give
  \[
   H^{2}\cdot\operatorname{ch}_{1}(Q_{i_{0}})=H^{2}\cdot\operatorname{ch}_{1}(\mathcal{H}^{0}(Q_{i_{0}}))-H^{2}\cdot\operatorname{ch}_{1}(\mathcal{H}^{-1}(Q_{i_{0}}))\geq-b_{1}(\ell) H^{3}.
  \]
  The upper bound and the characterization of the equality case follow by additivity.

  Now, suppose $\ell^{\prime}$ is a wall of $[Q_{i}]$ for some $i\neq i_{0}$.
  Suppose, for contradiction, that there exists $w^{\prime}\geq w$ such that $(b,w^{\prime})\in \ell^{\prime}$.
  By the same argument used above to bound $b_{1}(\ell)$ and $b_{2}(\ell)$, we find
  \[
   (b_{2}(\ell^{\prime})-b_{1}(\ell^{\prime}))H^{3}\leq H^{2}\cdot\operatorname{ch}_{1}(Q_{i}).
  \]
  Since $\ell^{\prime}$ is a wall of the rank zero class $[Q_{i}]$, the supporting numerical line is parallel to $\ell$.
  Therefore,
  \[
   (b_{2}(\ell^{\prime})-b_{1}(\ell^{\prime}))^{2}=(b_{2}(\ell)-b_{1}(\ell))^{2}+8(w^{\prime}-w).
  \]
  By additivity and the lower bound on
$H^{2}\cdot\operatorname{ch}_{1}(Q_{i_0})$, we obtain a contradiction with the assumed bound on the width of $\ell$.
 \end{proof}
 
 The following lemma adapts \cite{feyzbakhsh2024}*{Lemmas 2.6 and B.3} by replacing the BMT conjecture with Lemma \ref{lemma:numericalWallBound}.
 
 \begin{lemma}[cf. \cite{feyzbakhsh2024}*{Lemma B.3}]
  \label{lemma:rankZeroThirdChernBound}
  Let $E$ be a nonzero pure sheaf supported in codimension one.
  If $E$ is $\nu_{b,w}$-semistable for some $(b,w)\in U$ then
  \[
   \operatorname{ch}_{3}(E)\leq\frac{(H\cdot\operatorname{ch}_{2}(E))^{2}}{2H^{2}\cdot\operatorname{ch}_{1}(E)}+\frac{(H^{2}\cdot\operatorname{ch}_{1}(E))^{3}}{24 (H^{3})^{2}}+O(H^{2}\cdot\operatorname{ch}_{1}(E)).
  \]
  Here, the implied constant depends only on $(X,H)$.
 \end{lemma}
 \begin{proof}
  Set
  \[
   C=\max\left\{t_{0},1\right\}+\frac{(H^{2}\cdot\operatorname{ch}_{1}(E))^{2}}{8(H^{3})^{2}}
  \]
  where $t_{0}$ is as in Lemma \ref{lemma:numericalWallBound}.
  Consider the numerical line $\ell$ for $[E]$ passing through $(\nu,\tfrac{1}{2}\nu^{2}+C)$ where $\nu=H\cdot\operatorname{ch}_{2}(E)/(H^{2}\cdot\operatorname{ch}_{1}(E))$.
  Since $E$ is a sheaf supported in codimension $\geq 1$ and is $\nu_{b,w}$-semistable, $E$ is $\nu_{b,w^{\prime}}$-semistable for all $w^{\prime}\geq w$.
  Therefore, there exists $(b^{\prime},w^{\prime})$ above $\ell$ where $E$ is $\nu_{b^{\prime},w^{\prime}}$-semistable.
  By a standard argument (cf. \cite{li2019}*{Proof of Lemma 4.10}), if $\ell^{\prime}$ is a destabilizing wall for $E$ then
  \begin{equation}
   \label{equation:lisBound}
   b_{2}(\ell^{\prime})-b_{1}(\ell^{\prime})\leq\frac{H^{2}\cdot\operatorname{ch}_{1}(E)}{H^{3}}.
  \end{equation}
  By construction,
  \[
   b_{2}(\ell)-b_{1}(\ell)=\sqrt{8C}>\frac{H^{2}\cdot\operatorname{ch}_{1}(E)}{H^{3}}
  \]
  and so every destabilizing wall of $E$ lies below $\ell$.
  In particular, we can choose $(b^{\prime},w^{\prime})\in\ell$ where $E$ is $\nu_{b^{\prime},w^{\prime}}$-semistable.
  Therefore, since $(b_{2}(\ell)-b_{1}(\ell))^{2}=8C\geq 8t_{0}$, by Lemma \ref{lemma:numericalWallBound} applied to $\ell$,
  \[
   \operatorname{ch}_{3}(E)-\frac{2\nu}{3}H\cdot\operatorname{ch}_{2}(E)+\frac{\nu^{2}-2C}{6}H^{2}\cdot\operatorname{ch}_{1}(E)\leq MH^{2}\cdot\operatorname{ch}_{1}(E).
  \]
  Substituting the expressions for $\nu$ and $C$ and expanding gives the claimed bound.
 \end{proof}
 
 \begin{lemma}[cf. \cite{feyzbakhsh2024}*{Lemma 2.6}]
   \label{lemma:rankMinusOneThirdChernBound}
  Assume $E\in D^{b}(X)$ satisfies $\operatorname{ch}_{0}(E)=-1$ and $H^{2}\cdot\operatorname{ch}_{1}(E)=0$.
  If $E$ is $\nu_{1/H^{3},w}$-semistable for some $(1/H^{3},w)\in U$ then $H\cdot\operatorname{ch}_{2}(E)\geq 0$ and
  \[
   \operatorname{ch}_{3}(E)\leq\frac{2}{3}(H\cdot\operatorname{ch}_{2}(E))^{2}+O(1+H\cdot\operatorname{ch}_{2}(E)).
  \]
  Here, the implied constant depends only on $(X,H)$.
 \end{lemma}
 \begin{proof}
  Since $E$ is $\nu_{1/H^{3},w}$-semistable, $\Delta(E)\geq 0$.
  Therefore, since $\operatorname{ch}_{0}(E)=-1$ and $H^{2}\cdot\operatorname{ch}_{1}(E)=0$, it follows that $H\cdot\operatorname{ch}_{2}(E)\geq 0$.
  
  Choose $C>0$ satisfying $C>1/H^{3}$ and $C^{2}\geq 2 t_{0}$ where $t_0$ is as in Lemma \ref{lemma:numericalWallBound}.
  Consider the numerical line $\ell$ for $[E]$ passing through $(1/H^{3},C/H^{3})$.
  The denominator of $\nu_{1/H^{3},w}$ is $H^{2}\cdot\operatorname{ch}_{1}-\operatorname{ch}_{0}$.
  Thus, since $H^{2}\cdot\operatorname{ch}_{1}(E)-\operatorname{ch}_{0}(E)=1$ is the smallest positive integer, $E$ is $\nu_{1/H^{3},C/H^{3}}$-semistable.
  By direct computation, $(b_{2}(\ell)-b_{1}(\ell))^{2}\geq 8t_{0}$ so, by Lemma \ref{lemma:numericalWallBound},
  \[
   \operatorname{ch}_{3}(E)\leq\frac{2}{3}(H\cdot\operatorname{ch}_{2}(E)+C)H\cdot\operatorname{ch}_{2}(E)+M(H\cdot\operatorname{ch}_{2}(E)+C)H^{3},
  \]
  as desired.
 \end{proof}
 
\subsection{Replacing BMT estimates}
 The estimates in \cite{feyzbakhsh2023} that rely on the BMT conjecture occur in their Equation (23), Proposition 2.4, and Proposition 2.6.
 We show that these estimates can be replaced using the preceding results and explain the resulting modification of their Proposition 2.5.
 In this subsection, we follow the notation of \cite{feyzbakhsh2023}*{Section 2}.
 
 This subsection is intended to be read alongside \cite{feyzbakhsh2024}*{Section 2}.
 The author plans to give a self-contained exposition of this argument in a future version of this write-up.
 
 Fix $v\in K_{\mathrm{num}}(X)$ with $\operatorname{ch}_{0}(v)=0$ and $H^{2}\cdot\operatorname{ch}_{1}(v)>0$.
 For $n_{0}\gg 0$ set $v_{n_{0}}=v-[\mathcal{O}_{X}(-n_{0}H)]$.
 We also write
 \[
  \frac{1}{H^{3}}\left(H^{3}\operatorname{ch}_{0}(v),H^{2}\cdot\operatorname{ch}_{1}(v),H\cdot\operatorname{ch}_{2}(v),\operatorname{ch}_{3}(v)\right)=(0,c,s_{0},d_{0})
 \]
 and
 \[
  \frac{1}{H^{3}}\left(H^{3}\operatorname{ch}_{0}(v_{n_{0}}),H^{2}\cdot\operatorname{ch}_{1}(v_{n_{0}}),H\cdot\operatorname{ch}_{2}(v_{n_{0}}),\operatorname{ch}_{3}(v_{n_{0}})\right)=(-1,n_{0}+c,s_{0}-\tfrac{1}{2}n_{0}^{2},d_{0}+\tfrac{1}{6}n_{0}^{3}).
 \]
 
\subsubsection{Equation 23}
 Replace the line $\ell_{f}$ of \cite{feyzbakhsh2023}*{Equation (23)} with the numerical line for $v_{n_{0}}$ of slope
 \[
  \lambda_{n_{0}}=\frac{3(d_{0}+\frac{1}{6}n_{0}^{3})+(n_{0}+c)(s_{0}-\frac{1}{2}n_{0}^{2})-3M(n_{0}+c)}{(n_{0}+c)^{2}+2(s_{0}-\frac{1}{2}n_{0}^{2})+3M}.
 \]
 where $M$ is as in Lemma \ref{lemma:numericalWallBound}.
 Direct computation gives
 \begin{gather*}
  \lambda_{n_{0}}=-\frac{n_{0}}{4}+\frac{c^{2}+6s_{0}-9M}{8c}+O(n_{0}^{-1}) \\
  b_{1}(\ell_{f})=-n_{0}+\tfrac{1}{3}c+O(n_{0}^{-1}) \\
  b_{2}(\ell_{f})=\tfrac{1}{2}n_{0}+O(1)
 \end{gather*}
 These identities ensure that \cite{feyzbakhsh2023}*{Lemmas 2.1 and 2.2} remain valid above $\ell_{f}$.

\subsubsection{Proposition 2.4}
 In the proof of \cite{feyzbakhsh2023}*{Proposition 2.4}, Lemma \ref{lemma:rankZeroThirdChernBound} replaces Equation (38) with
 \[
  d-d^{\prime}\leq\frac{(s-s^{\prime})^{2}}{2(c-c^{\prime})}+\frac{(c-c^{\prime})^{3}}{24}+C_{0}(c-c^{\prime}).
 \]
 where $C_{0}>0$ depends only on $(X,H)$.
 
 The width estimate of their Equation (44) gives $b_{2}(\ell)-b_{1}(\ell)\geq\tfrac{3}{2}n_{0}+O(1)$, so  Lemma \ref{lemma:numericalWallBound} replaces the BMT application preceding their Equation (40).
 In turn, their Equation (40) is modified by adding
 \[
  3M(n+c^{\prime})\left(n+c^{\prime}+\frac{s-s^{\prime}}{c-c^{\prime}}\right)
 \]
 to the left side and subtracting
 \[
  3 C_{0}(n+c^{\prime})(c-c^{\prime})
 \]
 from the final right-hand side.
 When $c^{\prime}\leq 0$, these corrections are $O(n_{0}^{2})$ and $O(n_{0})$ respectively.
 Thus, the leading $n_{0}^{3}$ terms of Equations (41) and (42) are unchanged.
 As such, the conclusion of Proposition 2.4 holds.
 
\subsubsection{Proposition 2.5}
 For $n=n_{0}$, $c^{\prime}=0$, $s=s_{0}$, and $d=d_{0}$, the modification of Equation (40) only gives $s^{\prime}=O(1)$ not $s^{\prime}=0$.
 We give a different argument to show $s^{\prime}=0$.
 Equation (35) shows $E_{1}(n_{0})$ is $\nu_{1/H^{3},w}$-semistable for some $(1/H^{3},w)\in U$.
 Therefore, by Lemma \ref{lemma:rankMinusOneThirdChernBound} and the modified Equation (38), we find
 \[
  \frac{\operatorname{ch}_{3}(E_{1}(n_{0}))}{H^{3}}=d^{\prime}+n_{0}s^{\prime}\leq O(1)
 \]
 where $d^{\prime}\geq -O(1)$.
 In particular, $0\leq n_{0} s^{\prime}\leq O(1)$.
 Since $s^{\prime}\in (2 H^{3})^{-1}\mathbb{Z}$, $s^{\prime}=0$ for $n_{0}\gg 0$.
 By \cite{feyzbakhsh2024}*{Proof of Lemma 2.6}, $E_{1}\cong T(-n_{0}H)[1]$ where $T$ is a line bundle with $\operatorname{ch}_{1}(T)=0$, as needed.

\subsubsection{Proposition 2.6}
 Define $\delta_{n}=n_{0}-n$ and choose
 \begin{gather*}
  A>\max\left\{0,|s_{0}|H^{3},\frac{9M-c^{2}-6 s_{0}}{8 c}\right\}, \\
  B\geq\tfrac{1}{2} A^{2}+C_{0}
 \end{gather*}
 Then, $\lambda_{n_{0}}>-\tfrac{1}{4}n_{0}-A$ for $n_{0}\gg 0$.
 In \cite{feyzbakhsh2023}*{Definition 2.3}, use our version of $\ell_{f}$ and replace the upper bound for $s$ and the lower bound for $d$ by
 \begin{gather*}
  s\leq -\frac{3}{4}n_{0}\delta_{n}+(c+A)\delta_{n}+s_{0} \\
  d\geq\left(\frac{15n_{0}^{2}}{32}-n_{0}(A+c)-B\right)\delta_{n}+d_{0}.
 \end{gather*}
 Otherwise, keep the other requirements of that definition the same.
 
 Since $B\geq\tfrac{1}{2}A^{2}+C_{0}$, using our modified Equation (38), their calculation following Equation (48) gives
 \[
  \widetilde{d}\geq d+(\tfrac{15}{32}n_{0}^{2}-n_{0}(A+c)-B)(c-c^{\prime}).
 \]
 Our lower bound for $d$ gives the necessary lower bound for $\widetilde{d}$.
 The remaining estimates are unchanged, so the conclusions of their Proposition 2.6 remain valid.
 
\subsection{Proof of Theorem C}
 We give the reduction of rank zero invariants to rank one invariants in the large volume limit.

 \begin{proposition}[cf. \cite{feyzbakhsh2023}*{Theorem 3.1}]
  \label{proposition:invariantsAtLargeVolume}
  Suppose $(X,H)$ is a polarized Calabi--Yau threefold.
  If $v\in K_{\mathrm{num}}(X)$ with $\operatorname{ch}_{0}(v)=0$ and $H^{2}\cdot\operatorname{ch}_{1}(v)>0$, then $\mathsf{J}_{0,\infty}(v)$ is a polynomial with rational coefficients (depending only on $v$, $(X,H)$, and the graded ring $H^{\bullet}(X,\mathbb{Q})$) in finitely many invariants $\mathsf{J}_{0,\infty}(\beta_{i})$ where $\operatorname{ch}_{0}(\beta_{i})=-1$ and $H^{2}\cdot\operatorname{ch}_{1}(\beta_{i})>0$.
 \end{proposition}
 \begin{proof}
  Choose $n_{0}\gg 0$ so the preceding results hold and apply the proof of \cite{feyzbakhsh2023}*{Theorem 3.1} at $b=0$.
  The modifications of Propositions 2.4, 2.5, and 2.6 recover the wall description of \cite{feyzbakhsh2023}*{Section 2.1}.
  The argument of \cite{feyzbakhsh2023}*{Lemma 2.8} is unchanged, so their Proposition 3.2 still applies with our modifications.
  
  For $(0,w)$ just below $\ell_{f}$, the numerical line for $v_{n_{0}}$ through $(0,w)$ satisfies
  \[
   (b_{2}(\ell)-b_{1}(\ell))^{2}=4\nu_{0,w}(v_{n_{0}})^{2}+8w\geq\tfrac{4}{3}n_{0}^{2}.
  \]
  Thus, for $n_{0}\gg 0$, Lemma \ref{lemma:numericalWallBound} and our definition of $\ell_{f}$ imply $\mathsf{J}_{0,w}(v_{n_{0}})=0$.
  This supplies the starting chamber of \cite{feyzbakhsh2023}*{Equation (62)}.
  The remainder of the wall-crossing argument in \cite{feyzbakhsh2023}*{Section 3} now applies unchanged.
 \end{proof}

 Comparing $\mathsf{J}_{0,\infty}(v)$ with $\mathsf{J}(v)$ gives rank one reduction.
 
 \begin{theorem}
  \label{theorem:rankOneReduction}
  Suppose $(X,H)$ is a polarized Calabi--Yau threefold.
  For all $v\in K_{\mathrm{num}}(X)$ with $\operatorname{ch}_{0}(v)\geq 0$ there exist classes $\beta_{1},\beta_{2},\ldots,\beta_{m}\in K_{\mathrm{num}}(X)$ with $\operatorname{ch}_{0}(\beta_{i})=1$ and a polynomial $P_{v}\in\mathbb{Q}[t_{1},t_{2},\ldots,t_{m}]$ (depending on $v$, $(X,H)$, and the graded ring $H^{\bullet}(X,\mathbb{Q})$) such that
  \[
   \mathsf{J}(v)=P_{v}\left(\mathsf{J}(\beta_{1}),\mathsf{J}(\beta_{2}),\ldots,\mathsf{J}(\beta_{m})\right).
  \]
 \end{theorem}
 \begin{proof}
  If no coherent sheaf has numerical class $v$, then $\mathsf{J}(v)=0$ and the claim is immediate.
  By Corollary \ref{corollary:rankZeroReduction}, we may assume $\operatorname{ch}_{0}(v)=0$.
  
  If $H^{2}\cdot\operatorname{ch}_{1}(v)=0$ then every sheaf of numerical class $v$ is supported in dimension at most one.
  The case $H\cdot\operatorname{ch}_{2}(v)>0$ is treated in \cite{feyzbakhsh2023}*{Section 4} (which does not assume the BMT conjecture).
  If $H\cdot\operatorname{ch}_{2}(v)=0$ then every sheaf of numerical class $v$ has zero dimensional support.
  In this case, \cite{joyce2012}*{Equation (6.19)} expresses $\mathsf{J}(v)$ as a rational multiple of $\chi_{\mathrm{top}}(X)$, as needed.
  
  Thus, we may assume $H^{2}\cdot\operatorname{ch}_{1}(v)>0$.
  By Proposition \ref{proposition:invariantsAtLargeVolume}, we can write $\mathsf{J}_{0,\infty}(v)$ as a rational polynomial in finitely many invariants $\mathsf{J}_{0,\infty}(\beta_{i})$ with $\operatorname{ch}_{0}(\beta_{i})=-1$ and $H^{2}\cdot\operatorname{ch}_{1}(\beta_{i})>0$.
  By \cite{feyzbakhsh2023}*{Appendix A} and \cite{bridgeland2011}*{Theorem 1.1}, each $\mathsf{J}_{0,\infty}(\beta_{i})$ is a rational polynomial in finitely many DT invariants $\mathsf{J}(\alpha_{i})$ with $\operatorname{ch}_{0}(\alpha_{i})=1$.
  Hence, $\mathsf{J}_{0,\infty}(v)$ is also a rational polynomial in finitely many rank-one DT invariants.
  Further, by \cite{feyzbakhsh2023}*{Equation (66)}, we have
  \[
   \mathsf{J}(v)=\mathsf{J}_{0,\infty}(v)+\sum_{i} c_{i}\prod_{j}\mathsf{J}_{0,\infty}(\alpha_{i,j})
  \]
  where only finitely many $c_{i}$ are nonzero and $\alpha_{i,j}\in K_{\mathrm{num}}(X)$ satisfy $\operatorname{ch}_{0}(\alpha_{i,j})=0$ and $H^{2}\cdot\operatorname{ch}_{1}(\alpha_{i,j})>0$.
  By the same argument, each of the finitely many invariants $\mathsf{J}_{0,\infty}(\alpha_{i,j})$ can be expressed in terms of rank one DT invariants, as desired.
 \end{proof}

\end{document}